\documentclass[final,3p]{elsarticle}
\usepackage[latin1]{inputenc}
 \usepackage{graphics}
 \usepackage{graphicx}
 \usepackage{epsfig}
\usepackage{amssymb}
 \usepackage{amsthm}
 \usepackage{lineno}
 \usepackage{amsmath}
\usepackage{color}
   \numberwithin{equation}{section}
\usepackage{mathrsfs}

\journal{``"} 

\newtheorem{thm}{Theorem}[section]

\newtheorem{lem}[thm]{Lemma}

\newtheorem{defn}[thm]{Definition}

\biboptions{sort&compress,square}
\allowdisplaybreaks
\begin{document}
\begin{frontmatter}
\author{Tong Wu$^{a}$}
\ead{wut977@nenu.edu.cn}
\author{Yong Wang$^{b,*}$}
\ead{wangy581@nenu.edu.cn}
\cortext[cor]{Corresponding author.}
\address{$^a$Department of Mathematics, Northeastern University, Shenyang, 110819, China}
\address{$^b$School of Mathematics and Statistics, Northeast Normal University,
Changchun, 130024, China}

\title{The spectral Einstein functional for the Witten Deformation}
\begin{abstract}
 In the paper, given two vector fields and the Witten deformation, we compute the spectral Einstein functional for the Witten deformation on even-dimensional Riemannian manifolds without boundary.
\end{abstract}
\end{frontmatter}
\section{Introduction}
 Until now, many geometers have studied the noncommutative residues. The noncommutative residues are of great importance to the study of noncommutative geometry. Connes showed us that the noncommutative residue on a compact manifold $M$ coincided with the Dixmier's trace on pseudodifferential operators of order $-{\rm {dim}}M$ in \cite{Co2}. Therefore, the non-commutative residue can be used as integral of noncommutative geometry and become an important tool of noncommutative geometry. In \cite{Co1}, Connes used the noncommutative residue to derive a conformal 4-dimensional Polyakov action analogy.
Several years ago, Connes made a challenging observation that the noncommutative residue of the square of the inverse of the Dirac operator was proportional to the Einstein-Hilbert action, which we call the Kastler-Kalau-Walze theorem. In \cite{Ka}, Kastler gave a bruteforce proof of this theorem. In \cite{KW}, Kalau and Walze proved this theorem in the normal coordinates system simultaneously. Based on the theory of noncommutative residues proposed by Wodzicki, Fedosov, etc. \cite{FGLS} constructed the noncommutative residues of the classical elemental algebra of the Boutte de Monvel calculus on compact manifolds of dimension $n>2$.

Using elliptic pseudo-differential operators and noncommutative residues is a natural way to study the spectral Einstein functional and the operator-theoretic interpretation of the gravitational action on bounded manifolds. Concerning the Dirac operator and the signature operator, Wang carried out the computation of noncommutative residues and succeeded in proving the Kastler-Kalau-Walze type theorem for manifolds with boundary \cite{Wa1, Wa3, Wa4}. In \cite{FGV2}, Figueroa, et al. introduced a noncommutative integral based on the noncommutative residue.  Pf${\mathrm{\ddot{a}}}$ffle and Stephan \cite{pf1} considered orthogonal connections with arbitrary torsion on
compact Riemannian manifolds and computed the spectral action. In \cite{DL}, Dabrowski, Sitarz and Zalecki defined the spectral Einstein functional for a general spectral triple and for the noncommutative torus, they computed the spectral Einstein functional.
 In \cite{WWw}, Wang etc. gave some new spectral functionals which is the extension of spectral functionals to the noncommutative realm with torsion, and related them to the noncommutative residue for manifolds with boundary about Dirac operators with torsion.
 In \cite{DL2}, Dabrowski, Sitarz and Zalecki examined the metric and Einstein bilinear functionals of differential forms for the Hodge-Dirac operator $d+d^*$ on an oriented, closed, even-dimensional Riemannian manifold. Hong and Wang computed the spectral Einstein functional associated with the Dirac operator with torsion
 on even-dimensional spin manifolds without boundary in \cite{Hj}. Witten introduced an elliptic differential operator-Witten deformation in order to prove the Morse inequality (see\cite{ZW}). In \cite{B3}, Bao, Shan and Sun proved the general Kastler-Kalau-Walze type theorem about Witten deformation for any even
dimensional manifolds with boundary. {\bf The motivation} of this paper is to compute the spectral Einstein functional for the Witten deformation $d+d^*+\hat{c}(V)$ on even-dimensional Riemannian manifolds without boundary, where $d,d^*,V$ are defined in Section \ref{section:2}. Our main theorem is as follows.
 \begin{thm}\label{1thm}
 	Let $M$ be an $n=2m$ dimensional ($n\geq 3$) Riemannian manifold, for the  Witten deformation $\widetilde{D}_v$, the metric functional $\mathscr{M}_{\widetilde{D}_v}$ and the spectral Einstein functional $\mathscr{S}_{\widetilde{D}_v}$ are equal to
 	
 \begin{align}
 	\mathscr{M}_{\widetilde{D}_v}&=\mathrm{Wres}\bigg(c(u)c(w){\widetilde{D}_v^{-2m}}\bigg)=-2^{2m} \frac{2 \pi^{m}}{\Gamma\left(m\right)}\int_{M} g(u,w)d{\rm Vol}_M;\\
 	\mathscr{S}_{\widetilde{D}_v}&=\mathrm{Wres}\bigg(c(u)\big(c(w)\widetilde{D}_v+\widetilde{D}_vc(w)\big)\widetilde{D}_v^{-2m+1}\bigg)=\;2^{2m} \frac{2 \pi^{m}}{\Gamma\left(m\right)}\int_{M}\bigg(-\frac{1}{6}\mathbb{G}(u,w)+|V|^2g(u,w)\bigg)d{\rm Vol}_M,
 \end{align}
 where  $g(u,w)=\sum_{a,b=1}^{n}u_{a} w_{b} $ and $ \mathbb{G}(u,w)=\operatorname{Ric}(u,w)-\frac{1}{2} s g(u,w) $, $c(u)=\sum_{\eta=1}^{n} u_{\eta}c(e_{\eta}), c(w)=\sum_{\gamma=1}^{n} w_{\gamma}c(e_{\gamma}) .$
 \end{thm}

This paper is organized as follows. In Section \ref{section:2}, we introduce some notations about Clifford action and the Witten deformation, and symbols of the generalized laplacian for the Witten deformation are given. Using the residue for a differential operator of Laplace type ${\rm Wres}(P):=\int_{S^*M}{\rm tr}(\sigma_{-n}^P)(x,\xi)$and the composition formula of pseudo-differential operators, we obtain the metric functional $\mathscr{M}_{\widetilde{D}_v}$ and the spectral Einstein functional $\mathscr{S}_{\widetilde{D}_v}$ for the the Witten deformation on even-dimensional Riemannian manifolds without boundary in Section \ref{section:3}.
\section{Clifford Actions and the Witten Deformation}
\label{section:2}
Firstly, we introduce some notations about Clifford action and the Witten deformation, see in Section 4 in \cite{ZW}.

Let $M$ be an $n=2m$-dimensional ($n\geq 3$) oriented compact Riemannian manifold with a Riemannian metric $g^{TM}$, $TM$ (resp. $T^*M$) denote the tangent (resp. cotangent) vector bundle of $M$. And let $\nabla^L$ be the Levi-Civita connection about $g^{TM}$. In the
fixed orthonormal frame $\{e_1,\cdots,e_n\}$ in $TM$, the connection matrix $(\omega_{s,t})$ is defined by
\begin{equation}
\label{eq1}
\nabla^L(e_1,\cdots,e_n)= (e_1,\cdots,e_n)(\omega_{s,t}).
\end{equation}

 Let $c(e), \hat{c}(e)$ be the Clifford operators acting on the exterior algebra bundle $\Lambda^*(T^*M)$ of $T^*M$ defined by
 $$c(e)=e^*\wedge-i_e,~~~\hat{c}(e)=e^*\wedge+i_e,$$
where $e^*$ and $i_e$ are the notation for the exterior and interior multiplications respectively.
For  $\{e_1,\cdots,e_n\}$, one has
\begin{align}
\label{ali1}
&\hat{c}(e_i)\hat{c}(e_j)+\hat{c}(e_j)\hat{c}(e_i)=2g^{TM}(e_i,e_j);~~\nonumber\\
&c(e_i)c(e_j)+c(e_j)c(e_i)=-2g^{TM}(e_i,e_j);~~\nonumber\\
&c(e_i)\hat{c}(e_j)+\hat{c}(e_j)c(e_i)=0.
\end{align}

Let $\{e^1,\cdots,e^n\}$ is the corresponding dual basis in $T^*M$, for any integer $k$ between 1 and $n$, then the Hodge star operator
$\ast:\Lambda^*(T^*M)\rightarrow\Lambda^*(T^*M)$ can be defined as follows:
$$\ast:e^1\wedge\cdot\cdot\cdot\wedge e^k\mapsto e^{k+1}\wedge\cdot\cdot\cdot\wedge e^n.$$

From some properties of the Hodge star operator, one can define inner product $\langle \cdot, \cdot\rangle$ on $\Omega^*(M)$ as follows:
$$\langle \alpha, \beta\rangle=\int_M\alpha\wedge\ast\beta,$$
where $\Omega^*(M):=\Gamma(\Lambda^*(T^*M))$ denotes the space of smooth sections of $\Lambda^*(T^*M)$.

Recall that $d:\Omega^*(M)\rightarrow\Omega^*(M)$ is the exterior differential operator on $M$.
\begin{defn}\cite{ZW}
Let $d^*:\Omega^*(M)\rightarrow\Omega^*(M)$ be the operator defined by
\begin{align}
\alpha\in\Omega^k(M)\mapsto(-1)^{nk+n+1}\ast d\ast \alpha \in\Omega^{k-1}(M),
\end{align}
where $\Omega^k(M):=\Gamma(\Lambda^p(T^*M))$ denotes the space of smooth p-forms over on $M$.
\end{defn}
For $\alpha\in \Omega^k(M),\beta\in \Omega^{k+1}(M)$, which satisfies $\langle d\alpha, \beta\rangle=\langle \alpha, d^*\beta\rangle.$ That is, $d^*$ is the formal adjoint of $d$.
\begin{defn}\cite{ZW}
The de Rham-Hodge operator $\widetilde{D}$ associated to $g^{TM}$ is defined by
\begin{align}
\widetilde{D}=d+d^*=\sum^n_{i=1}c(e_i)\nabla^{\Lambda^*(T^*M)}_{e_i}=\sum^n_{i=1}c(e_i)\bigg[e_i+\frac{1}{4}\sum_{s,t}\omega_{s,t}
(e_i)[\widehat{c}(e_s)\widehat{c}(e_t)
-c(e_s)c(e_t)]\bigg].
\end{align}
\end{defn}
Following \cite{Wi}, for a vector field $V$, set the Witten deformation
$$\widetilde{D}_v=d+d^*+\hat{c}(V)=\sum^n_{i=1}c(e_i)\nabla^{\Lambda^*(T^*M)}_{e_i}+\hat{c}(V)=\sum^n_{i=1}c(e_i)\bigg[e_i+\frac{1}{4}\sum_{s,t}\omega_{s,t}
(e_i)[\widehat{c}(e_s)\widehat{c}(e_t)
-c(e_s)c(e_t)]\bigg]+\hat{c}(V).$$

For a differential operator of Laplace type $P$, it has locally the form
\begin{equation}\label{p}
	P=-(g^{ij}\partial_i\partial_j+A^i\partial_i+B),
\end{equation}
where $\partial_{i}$  is a natural local frame on $TM,$ $(g^{ij})_{1\leq i,j\leq n}$ is the inverse matrix associated to the metric
matrix  $(g_{ij})_{1\leq i,j\leq n}$ on $M,$ $A^{i}$ and $B$ are smooth sections of $\textrm{End}(N)$ on $M$ (endomorphism).
If $P$ satisfies the form \eqref{p}, then there is a unique
connection $\nabla$ on $N$ and a unique endomorphism $\widetilde{E}$ such that
\begin{equation}
	P=-[g^{ij}(\nabla_{\partial_{i}}\nabla_{\partial_{j}}- \nabla_{\nabla^{L}_{\partial_{i}}\partial_{j}})+\widetilde{E}].
\end{equation}

Moreover
(with local frames of $T^{*}M$ and $N$), $\nabla_{\partial_{i}}=\partial_{i}+\omega_{i} $
and $\widetilde{E}$ are related to $g^{ij}$, $A^{i}$ and $B$ through
\begin{eqnarray}
	&&\omega_{i}=\frac{1}{2}g_{ij}\big(A^{i}+g^{kl}\Gamma_{ kl}^{j} \texttt{id}\big),\nonumber\\
	&&\widetilde{E}=B-g^{ij}\big(\partial_{i}(\omega_{j})+\omega_{i}\omega_{j}-\omega_{k}\Gamma_{ ij}^{k} \big),
\end{eqnarray}
where $\Gamma_{ kl}^{j}$ is the  Christoffel coefficient of $\nabla^{L}$.

From \cite{DL} ,we get
\begin{align} \label{dt2}
\widetilde{D}_v
^{2}=-g^{ab}(\nabla_{{\partial}_a}\nabla_{{\partial}_b}-\nabla_{{\nabla_{{\partial}_a}^{L}}\partial_b})+E,
\end{align}
where $E=-\widetilde{E}, \nabla_{{\partial}_a}={\partial}_a-\widetilde{T}_a$.

By Propsition 4.6 in \cite{ZW}, the following identity holds:
$$\widetilde{D}_v
^{2}=\widetilde{D}^2+\sum_{i=1}^nc(e_i)\hat{c}(\nabla^{TM}_{e_i}V)+|V|^2.$$
\indent By \cite{Y}, the local expression of $\widetilde{D}^2$
\begin{equation}\label{eq8}
\widetilde{D}^2
=-\triangle+\frac{1}{8}\sum_{ijkl}R_{ijkl}\widehat{c}(e_i)\widehat{c}(e_j)c(e_k)c(e_l)+\frac{1}{4}s;
\end{equation}
\begin{equation}
\label{eq9}
-\triangle=-g^{ij}(\nabla^{\Lambda^*(T^*M)}_{i}\nabla^{\Lambda^*(T^*M)}_{j}-\Gamma_{ij}^{k}\nabla^{\Lambda^*(T^*M)}_{k}),
\end{equation}
where $R$ (resp. $s$) denote the Riemannian curvature (resp. scalar curvature).

Define $\omega_{s,t}
(e_i)=-\langle \nabla_{e_i}^{L}e_{s}, e_{t}\rangle $, we get
\begin{align} \label{ta1}
	\widetilde{T}_a&=-\frac{1}{4}\sum_{s, t=1}^{n}\langle \nabla_{{\partial}_a}^{L}e_{s}, e_{t}\rangle c(e_{s})c(e_{t})+\frac{1}{4}\sum_{s, t=1}^{n}\langle \nabla_{{\partial}_a}^{L}e_{s}, e_{t}\rangle \hat{c}(e_{s})\hat{c}(e_{t});\nonumber\\
E&=\frac{1}{8}\sum_{i,j,k,l=1}^nR_{ijkl}\hat{c}(e_i)\hat{c}(e_j)c(e_k)c(e_l)+\frac{1}{4}s+\sum_{i=1}^nc(e_i)\hat{c}(\nabla^{TM}_{e_i}V)+|V|^2.
\end{align}
In normal coordinates, $\widetilde{T}_a$ is expanded near $x=0$ by Taylor expansion. That is
\begin{align} \label{ta2}
	\widetilde{T}_a={T}_a+{T}_{ab}x^b+O(x^2).
\end{align}

By ${\partial}_{l}\langle \nabla_{{\partial}_a}^{L}e_{s}, e_{t}\rangle(x_0)=\frac{1}{2}{\rm R}_{lats}(x_0)$, we get
\begin{align} \label{ta0}
	{T}_a&=0;\nonumber\\
	{T}_{ab}&=-\frac{1}{8}\sum_{st}R_{bats}(x_0) c(e_{s})c(e_{t})+\frac{1}{8}\sum_{st}R_{bats}(x_0) \hat{c}(e_{s})\hat{c}(e_{t}).
\end{align}

 \begin{lem}\label{lem1}\cite{DL}
	The leading symbols of the generalized laplacian $\Delta_{T,E}^{-m}$ are the followings:
	\begin{align}\label{sigma0}
		\sigma_{-2 m}(\Delta_{T,E}^{-m})=&\|\xi\|^{-2 m-2}\sum_{a,b,j,k=1}^{2m}\left(\delta_{a b}-\frac{m}{3} R_{a j b k} x^{j} x^{k}\right) \xi_{a} \xi_{b}+O\left(\mathbf{x}^{3}\right) ;\\
		\sigma_{-2m-1}(\Delta_{T,E}^{-m})=& \frac{-2 m i}{3}\|\xi\|^{-2 m-2} \sum_{a,k=1}^{2m}\operatorname{Ric}_{a k} x^{k} \xi_{a}-2 m i\|\xi\|^{-2 m-2}\sum_{a,b=1}^{2m}\left(T_{a} \xi_{a}+T_{a b} x^{b} \xi_{a}\right)+O(\mathbf{x^2}) ;\\
		\sigma_{-2m-2}(\Delta_{T,E}^{-m})=& \frac{m(m+1)}{3}\|\xi\|^{-2 m-4} \sum_{a,b=1}^{2m}\operatorname{Ric}_{a b} \xi_{a} \xi_{b} \\
		&-2 m(m+1)\|\xi\|^{-2 m-4}\sum_{a,b=1}^{2m} T_{a} T_{b} \xi_{a} \xi_{b}+m\sum_{a,b=1}^{2m}\left(T_{a} T_{a}-T_{a a}\right)\|\xi\|^{-2 m-2}\nonumber \\
		&+2 m(m+1)\|\xi\|^{-2 m-4} \sum_{a,b=1}^{2m}T_{a b} \xi_{a} \xi_{b}-mE |\xi\|^{-2m-2}+O(\mathbf{x}) ,\nonumber
	\end{align}
where $Ric$ denotes Ricci curvature.
\end{lem}

By \eqref{ta1}, \eqref{ta0} and lemma \ref{lem1}, we get the following lemma.

\begin{lem}\label{lem2}
	The general dimensional symbols of the Witten deformation are given:
	\begin{align}\label{sigma}
		\sigma_{-2 m}(\Delta_{T,E}^{-m})=&\|\xi\|^{-2 m-2}\sum_{a,b,j,k=1}^{2m} \left(\delta_{a b}-\frac{m}{3} R_{a j b k} x^{j} x^{k}\right) \xi_{a} \xi_{b}+O\left(\mathbf{x}^{3}\right); \\
		\sigma_{-2m-1}(\Delta_{T,E}^{-m})=& \frac{-2 mi}{3}\|\xi\|^{-2 m-2}\sum_{a,b=1}^{2m}  \operatorname{Ric}_{a b} x^{b} \xi_{a}\\
		&+\frac{ m i}{4}\|\xi\|^{-2 m-2} \sum_{a,b,t,s=1}^{2m} \operatorname{R}_{b a t s}(x_0) c(e_s) c(e_t) x^{b} \xi_{a}\nonumber\\
		&-\frac{ m i}{4}\|\xi\|^{-2 m-2} \sum_{a,b,t,s=1}^{2m} \operatorname{R}_{b a t s}(x_0) \hat{c}(e_s) \hat{c}(e_t) x^{b} \xi_{a} +O\left(\mathbf{x}^{2}\right);\nonumber\\
		\sigma_{-2m-2}(\Delta_{T,E}^{-m})=& \frac{m(m+1)}{3}\|\xi\|^{-2 m-4}\sum_{a,b=1}^{2m} \operatorname{Ric}_{a b} \xi_{a} \xi_{b} \\
		&-\frac{ m (m+1)}{4}\|\xi\|^{-2 m-4} \sum_{a,b,t,s=1}^{2m}\operatorname{R}_{b a t s}(x_0) c(e_s) c(e_t) \xi_{a} \xi_{b}\nonumber\\
&+\frac{ m (m+1)}{4}\|\xi\|^{-2 m-4} \sum_{a,b,t,s=1}^{2m}\operatorname{R}_{b a t s}(x_0) \hat{c}(e_s) \hat{c}(e_t) \xi_{a} \xi_{b}\nonumber\\
		&-\frac{1}{8}m\|\xi\|^{-2 m-2}\sum_{ijkl}R_{ijkl}\hat{c}(e_i) \hat{c}(e_j)c(e_k) c(e_l)-
\frac{1}{4}m \|\xi\|^{-2 m-2}s\nonumber\\
		&-m\|\xi\|^{-2 m-2}\sum_{i=1}^{2m}c(e_i)\hat{c}(\nabla^{TM}_{e_i}V)-m\|\xi\|^{-2 m-2}|V|^2+O\left(\mathbf{x}\right),\nonumber
	\end{align}
where $i$ in $mi$ represents the imaginary unit, and $i$ in other subscript positions represents the $i$-th.
\end{lem}

 \section{The spectral Einstein functional for the Witten Deformation}
 \label{section:3}
In this section, we want to obtain the metric functional and the spectral Einstein functional for the Witten deformation.

For a pseudo-differential operator  $P$, acting on sections of a vector bundle over an even n-dimensional compact Riemannian manifold  M , the analogue of the volume element in noncommutative geometry is the operator  $D^{-n}:=d s^{n} $. And pertinent operators are realized as pseudo-differential operators on the spaces of sections. Extending previous definitions by Connes \cite{co5}, a noncommutative integral was introduced in \cite{FGV2} based on the noncommutative residue \cite{wo2}, combine (1.4) in \cite{co4} and \cite{Ka}, using the definition of the residue:
\begin{align}\label{wers}
	\int P d s^{n}:=\operatorname{Wres} (P D^{-n}):=\int_{S^{*} M} \operatorname{tr}\left[\sigma_{-n}\left(P D^{-n}\right)\right](x, \xi),
\end{align}

where  $\sigma_{-n}\left(P D^{-n}\right) $ denotes the  $(-n)$th order piece of the complete symbols of  $P D^{-n} $,  $\operatorname{tr}$  as shorthand of trace.

Firstly, we review here technical tool of the computation, which are the integrals of polynomial functions over the unit spheres. By (32) in \cite{B1}, we define
\begin{align}
I_{S_n}^{\gamma_1\cdot\cdot\cdot\gamma_{2\bar{n}+2}}=\int_{|x|=1}d^nxx^{\gamma_1}\cdot\cdot\cdot x^{\gamma_{2\bar{n}+2}},
\end{align}
i.e. the monomial integrals over a unit sphere.
Then by Proposition A.2. in \cite{B1},  polynomial integrals over higher spheres in the $n$-dimesional case are given
\begin{align}
I_{S_n}^{\gamma_1\cdot\cdot\cdot\gamma_{2\bar{n}+2}}=\frac{1}{2\bar{n}+n}[\delta^{\gamma_1\gamma_2}I_{S_n}^{\gamma_3\cdot\cdot\cdot\gamma_{2\bar{n}+2}}+\cdot\cdot\cdot+\delta^{\gamma_1\gamma_{2\bar{n}+1}}I_{S_n}^{\gamma_2\cdot\cdot\cdot\gamma_{2\bar{n}+1}}],
\end{align}
where $S_n\equiv S^{n-1}$ in $\mathbb{R}^n$.

For $\bar{n}=0$, we have $I^0={\rm Vol}(S^{n-1})$=$\frac{2\pi^{\frac{n}{2}}}{\Gamma(\frac{n}{2})}$, we immediately get
\begin{align}
I_{S_n}^{\gamma_1\gamma_2}&=\frac{1}{n}{\rm Vol}(S^{n-1})\delta^{\gamma_1\gamma_2};\nonumber\\
I_{S_n}^{\gamma_1\gamma_2\gamma_3\gamma_4}&=\frac{1}{n(n+2)}{\rm Vol}(S^{n-1})[\delta^{\gamma_1\gamma_2}\delta^{\gamma_3\gamma_4}+\delta^{\gamma_1\gamma_3}\delta^{\gamma_2\gamma_4}+\delta^{\gamma_1\gamma_4}\delta^{\gamma_2\gamma_3}].
\end{align}

 \begin{thm}\label{thm}
 	Let $M$ be an $n=2m$ dimensional ($n\geq 3$) Riemannian manifold, the metric functional $\mathscr{M}_{\widetilde{D}_v}$ and the spectral Einstein functional $\mathscr{S}_{\widetilde{D}_v}$ are equal to
 	
 \begin{align}
 	\mathscr{M}_{\widetilde{D}_v}&=\mathrm{Wres}\bigg(c(u)c(w){\widetilde{D}_v^{-2m}}\bigg)=-2^{2m} \frac{2 \pi^{m}}{\Gamma\left(m\right)}\int_{M} g(u,w)d{\rm Vol}_M;\\
 	\mathscr{S}_{\widetilde{D}_v}&=\mathrm{Wres}\bigg(c(u)\big(c(w)\widetilde{D}_v+\widetilde{D}_vc(w)\big)\widetilde{D}_v^{-2m+1}\bigg)=\;2^{2m} \frac{2 \pi^{m}}{\Gamma\left(m\right)}\int_{M}\bigg(-\frac{1}{6}\mathbb{G}(u,w)+|V|^2g(u,w)\bigg)d{\rm Vol}_M,
 \end{align}
 where  $g(u,w)=\sum_{a,b=1}^{n}u_{a} w_{b} $ and $ \mathbb{G}(u,w)=\operatorname{Ric}(u,w)-\frac{1}{2} s g(u,w) $, $c(u)=\sum_{\eta=1}^{n} u_{\eta}c(e_{\eta}), c(w)=\sum_{\gamma=1}^{n} w_{\gamma}c(e_{\gamma}) .$
 \end{thm}

\begin{proof}
The proof of the $\mathscr{M}_{\widetilde{D}_v}$ formula for the measure function is obvious. Before we prove that the spectral Einstein function $\mathscr{S}_{\widetilde{D}_v}$, splitting it into two parts as follows:
\begin{align}
\mathscr{S}_{\widetilde{D}_v}&=\mathrm{Wres}\bigg(c(u)c(w){\widetilde{D}_v^{-2m+2}}\bigg)+\mathrm{Wres}\bigg(c(u)\widetilde{D}_vc(w) \widetilde{D}_v\widetilde{D}_v^{-2m}\bigg)\nonumber\\
&:=\mathscr{S}_{1}+\mathscr{S}_{2}.\nonumber
\end{align}

{\bf Part I)} $\mathscr{S}_{1}=\mathrm{Wres}\bigg(c(u)c(w){\widetilde{D}_v^{-2m+2}}\bigg)$.

By \eqref{wers}, we need to compute  $\int_{S^{*} M} \operatorname{tr}\left[\sigma_{-2 m}\left(c(u)c(w) \widetilde{D}_v^{-2 m+2}\right)\right](x, \xi) $. Based on the algorithm yielding the principal symbol of a product of pseudo-differential operators in terms of the principal symbols of the factors,  by lemma \ref{lem2}, we have
\begin{align}
	\sigma_{-2m}(\widetilde{D}_v^{-2 m+2})=& \frac{m(m-1)}{3}\|\xi\|^{-2 m-4}\sum_{a,b=1}^{2m} \operatorname{Ric}_{a b} \xi_{a} \xi_{b} \\
		&-\frac{ m (m-1)}{4}\|\xi\|^{-2 m-4} \sum_{a,b,t,s=1}^{2m}\operatorname{R}_{b a t s}(x_0) c(e_s) c(e_t) \xi_{a} \xi_{b}\nonumber\\
&+\frac{ m (m-1)}{4}\|\xi\|^{-2 m-4} \sum_{a,b,t,s=1}^{2m}\operatorname{R}_{b a t s}(x_0) \hat{c}(e_s) \hat{c}(e_t) \xi_{a} \xi_{b}\nonumber\\
		&-\frac{1}{8}(m-1)\|\xi\|^{-2 m-2}\sum_{i,j,k,l=1}^{2m}R_{ijkl}\hat{c}(e_i) \hat{c}(e_j)c(e_k) c(e_l)-
\frac{1}{4}(m-1) \|\xi\|^{-2 m-4}s\nonumber\\
		&-(m-1)\|\xi\|^{-2 m-2}\sum_{i=1}^{2m}c(e_i)\hat{c}(\nabla^{TM}_{e_i}V)-(m-1)\|\xi\|^{-2 m-2}|V|^2+O\left(\mathbf{x}\right).
\end{align}
Obviously, the general dimensional symbols of the $c(u)c(w)\widetilde{D}_v^{-2 m+2}$ are given:
\begin{align}\label{PD}
	\sigma_{-2 m}\left(c(u)c(w)\widetilde{D}_v^{-2 m+2}\right) =& \frac{m(m-1)}{3}\|\xi\|^{-2 m-4}\sum_{a,b=1}^{2m} \operatorname{Ric}_{a b} \xi_{a} \xi_{b}c(u)c(w) \\
		&-\frac{ m (m-1)}{4}\|\xi\|^{-2 m-4} \sum_{a,b,t,s=1}^{2m}\operatorname{R}_{b a t s}(x_0) c(u)c(w)c(e_s) c(e_t) \xi_{a} \xi_{b}\nonumber\\
&+\frac{ m (m-1)}{4}\|\xi\|^{-2 m-4} \sum_{a,b,t,s=1}^{2m}\operatorname{R}_{b a t s}(x_0) c(u)c(w)\hat{c}(e_s) \hat{c}(e_t) \xi_{a} \xi_{b}\nonumber\\
		&-\frac{1}{8}(m-1)\|\xi\|^{-2 m-2}\sum_{i,j,k,l=1}^{2m}R_{ijkl}c(u)c(w)\hat{c}(e_i) \hat{c}(e_j)c(e_k) c(e_l)\nonumber\\
		&-
\frac{1}{4}(m-1) \|\xi\|^{-2 m-4}sc(u)c(w)-(m-1)\|\xi\|^{-2 m-2}\sum_{i=1}^{2m}c(u)c(w)c(e_i)\hat{c}(\nabla^{TM}_{e_i}V)\nonumber\\
		&-(m-1)\|\xi\|^{-2 m-2}|V|^2c(u)c(w)+O\left(\mathbf{x}\right).\nonumber
\end{align}
Next, we integrate each of the above items respectively.
\begin{align}
	&{\bf(I-1)}~\int_{\|\xi\|=1}\operatorname{tr} \biggl\{\frac{m(m-1)}{3}\|\xi\|^{-2 m-2}\sum_{a,b=1}^{2m} \operatorname{Ric}_{a b} \xi_{a} \xi_{b}c(u)c(w)\biggr\}(x_0)\sigma(\xi)\nonumber\\
	&\;\;\;\;\;\;\;\;\;\;=-\frac{m-1}{6}s g(u,w){\rm tr}[id]{\rm Vol}(S^{n-1});\nonumber\\
	&{\bf(I-2)}\int_{\|\xi\|=1} {\rm tr}\biggl\{-\frac{ m (m-1)}{4}\|\xi\|^{-2 m-2} \sum_{a,b,t,s=1}^{2m}\operatorname{R}_{b a t s}c(u)c(w) c(e_s) c(e_t) \xi_{a} \xi_{b}\biggr\}(x_0)\sigma(\xi)\nonumber\\
	&\;\;\;\;\;\;\;\;\;\;=-\frac{ m-1}{8}\sum_{a,t,s=1}^{2m}\operatorname{R}_{a a t s}{\rm tr}\bigg(c(u)c(w) c(e_s) c(e_t)\bigg)(x_0){\rm Vol}(S^{n-1})  \nonumber\\
	&\;\;\;\;\;\;\;\;\;\;=0;\nonumber\\
	&{\bf(I-3)}\int_{\|\xi\|=1} {\rm tr}\biggl\{-\frac{ m (m-1)}{4}\|\xi\|^{-2 m-2} \sum_{a,b,t,s=1}^{2m}\operatorname{R}_{b a t s}c(u)c(w) \hat{c}(e_s) \hat{c}(e_t) \xi_{a} \xi_{b}\biggr\}(x_0)\sigma(\xi)\nonumber\\
	&\;\;\;\;\;\;\;\;\;\;=0;\nonumber\\
	&{\bf(I-4)}\int_{\|\xi\|=1} {\rm tr}\biggl\{-\frac{1}{8}(m-1)\|\xi\|^{-2 m-2}\sum_{i,j,k,l=1}^{2m}R_{ijkl}c(u)c(w)\hat{c}(e_i) \hat{c}(e_j)c(e_k) c(e_l)\biggr\}(x_0)\sigma(\xi)\nonumber\\
	&\;\;\;\;\;\;\;\;\;\;=-\frac{1}{8}(m-1)\sum_{ijkl}R_{ijkl}(x_0){\rm tr}[c(u)c(w)\hat{c}(e_i) \hat{c}(e_j)c(e_k) c(e_l)](x_0){\rm Vol}(S^{n-1})\nonumber\\
&\;\;\;\;\;\;\;\;\;\;=0;\nonumber\\
	&{\bf(I-5)}\int_{\|\xi\|=1} {\rm tr}\biggl\{-\frac{1}{4}(m-1) \|\xi\|^{-2 m-4}sc(u)c(w)\biggr\}(x_0)\sigma(\xi)\nonumber\\
	&\;\;\;\;\;\;\;\;\;\;=-\frac{1}{4}(m-1)s{\rm tr}[c(u)c(w)]{\rm Vol}(S^{n-1})\nonumber\\
&\;\;\;\;\;\;\;\;\;\;=\frac{1}{4}(m-1)sg(u,w){\rm tr}[id]{\rm Vol}(S^{n-1});\nonumber\\
	&{\bf(I-6)}\int_{\|\xi\|=1} {\rm tr}\biggl\{-(m-1)\|\xi\|^{-2 m-2}\sum_{i=1}^{2m}c(u)c(w)c(e_i)\hat{c}(\nabla^{TM}_{e_i}V)\biggr\}(x_0)\sigma(\xi)\nonumber\\
	&\;\;\;\;\;\;\;\;\;\;=-(m-1)\sum_{i=1}^{2m}{\rm tr}[c(u)c(w)c(e_i)\hat{c}(\nabla^{TM}_{e_i}V)]{\rm Vol}(S^{n-1})\nonumber\\
&\;\;\;\;\;\;\;\;\;\;=0;\nonumber\\
	&{\bf(I-7)}\int_{\|\xi\|=1} {\rm tr}\biggl\{-(m-1)\|\xi\|^{-2 m-2}|V|^2c(u)c(w)\biggr\}(x_0)\sigma(\xi)\nonumber\\
	&\;\;\;\;\;\;\;\;\;\;=-(m-1)|V|^2{\rm tr}[c(u)c(w)]{\rm Vol}(S^{n-1})\nonumber\\
&\;\;\;\;\;\;\;\;\;\;=(m-1)|V|^2g(u,w){\rm tr}[id]{\rm Vol}(S^{n-1}).
\end{align}

Therefore, we get
\begin{align}\label{zpdt}
	&\int_{\|\xi\|=1} {\rm tr}\biggl\{\sigma_{-2 m}\left(c(u)c(w)  \widetilde{D}_v^{-2 m+2}\right)\biggr\}(x_0)\sigma(\xi)\\
	&=\left(\frac{m-1}{12}s g(u,w)+(m-1)|V|^2 g(u,w)\right){\rm tr}[id]{\rm Vol}(S^{n-1}).\nonumber
\end{align}
Since ${\rm tr}[id]=2^{2m}$ and ${\rm Vol}(S^{n-1})=\frac{2 \pi^{m}}{\Gamma\left(m\right)}$, we obtain
\begin{align}\label{z1}
	\mathscr{S}_{1}=&\mathrm{Wres}\bigg(c(u)c(w){\widetilde{D}^{-2 m+2}_v}\bigg)\\
	=&\;2^{2m} \frac{2 \pi^{m}}{\Gamma\left(m\right)}\int_{M}\left(\frac{m-1}{12}s g(u,w)+(m-1)|V|^2 g(u,w)\right)d{\rm Vol}_M\nonumber.
\end{align}

{\bf Part II)} $\mathscr{S}_{2}=\mathrm{Wres}\bigg(c(u)\widetilde{D}_vc(w) \widetilde{D}_v \widetilde{D}_v^{-2m}\bigg)$.

 Write $c(u)\widetilde{D}_v:=\mathcal{A},c(w)\widetilde{D}_v:=\mathcal{B} .$
By \eqref{wers}, we need to compute  $\int_{S^{*} M} \operatorname{tr}\left[\sigma_{-2 m}\left(\mathcal{A} \mathcal{B} \widetilde{D}_v
^{-2 m}\right)\right](x, \xi) $. Based on the algorithm yielding the principal symbol of a product of pseudo-differential operators in terms of the principal symbols of the factors, we have
\begin{align}\label{ABD}
	\sigma_{-2 m}\left(\mathcal{A} \mathcal{B} \widetilde{D}_v^{-2 m}\right) & =\left\{\sum_{|\alpha|=0}^{\infty} \frac{(-i)^{|\alpha|}}{\alpha!} \partial_{\xi}^{\alpha}[\sigma(\mathcal{A} \mathcal{B})] \partial_{x}^{\alpha}\left[\sigma\left(\widetilde{D}_v^{-2 m}\right)\right]\right\}_{-2 m} \\
	& =\sigma_{0}(\mathcal{A} \mathcal{B}) \sigma_{-2 m}\left(\widetilde{D}_v^{-2 m}\right)+\sigma_{1}(\mathcal{A} \mathcal{B}) \sigma_{-2 m-1}\left(\widetilde{D}_v^{-2 m}\right)+\sigma_{2}(\mathcal{A} \mathcal{B}) \sigma_{-2 m-2}\left(\widetilde{D}_v^{-2 m}\right) \nonumber\\
	& +(-i) \sum_{j=1}^{2m} \partial_{\xi_{j}}\left[\sigma_{2}(\mathcal{A} \mathcal{B})\right] \partial_{x_{j}}\left[\sigma_{-2 m-1}\left(\widetilde{D}_v^{-2 m}\right)\right]+(-i) \sum_{j=1}^{2m} \partial_{\xi_{j}}\left[\sigma_{1}(\mathcal{A} \mathcal{B})\right] \partial_{x_{j}}\left[\sigma_{-2 m}\left(\widetilde{D}_v^{-2 m}\right)\right] \nonumber\\
	& -\frac{1}{2} \sum_{j ,l=1}^{2m} \partial_{\xi_{j}} \partial_{\xi_{l}}\left[\sigma_{2}(\mathcal{A} \mathcal{B})\right] \partial_{x_{j}} \partial_{x_{l}}\left[\sigma_{-2 m}\left(\widetilde{D}_v^{-2 m}\right)\right] .\nonumber
\end{align}

 \begin{lem}
The symbols of  $\mathcal{A}$  and  $\mathcal{B}$  are given:
\begin{align}
&\sigma_{1}(\mathcal{A})=\sqrt{-1}c(u)c(\xi); \nonumber\\
	&\sigma_{1}(\mathcal{B})=\sqrt{-1}c(w)c(\xi);\nonumber\\
&\sigma_{0}(\mathcal{A})=-\frac{1}{4} \sum_{p,s,t=1}^{2m} w_{st}(e_p)c(u)c(e_p)c(e_s)c(e_t)+\frac{1}{4} \sum_{p,s,t=1}^{2m} w_{st}(e_p)c(u)c(e_p)\hat{c}(e_s)\hat{c}(e_t)+c(u)\hat{c}(V);\nonumber\\
	&\sigma_{0}(\mathcal{B})=-\frac{1}{4} \sum_{p,s,t=1}^{2m} w_{st}(e_p)c(w)c(e_p)c(e_s)c(e_t)+\frac{1}{4} \sum_{p,s,t=1}^{2m} w_{st}(e_p)c(w)c(e_p)\hat{c}(e_s)\hat{c}(e_t)+c(w)\hat{c}(V).
\end{align}
 \end{lem}

Further, by the composition formula of pseudo-differential operators, we get the following lemma.

 \begin{lem}
	The symbols of $\mathcal{AB}$ are given:
	\begin{align}
		\sigma_{0}(\mathcal{AB})=&\sigma_{0}(\mathcal{A}) \sigma_{0}(\mathcal{B})+(-i) \partial_{\xi_{j}}\left[\sigma_{1}(\mathcal{A})\right] \partial_{x_{j}}\left[\sigma_{0}(\mathcal{B})\right]+(-i) \partial_{\xi_{j}}\left[\sigma_{0}(\mathcal{A})\right] \partial_{x_{j}}\left[\sigma_{1}(\mathcal{B})\right]\nonumber\\
		=&\;\;\;\frac{1}{16}\sum_{p,s,t,\hat{p},\hat{s},\hat{t}=1}^{2m} w_{st}(e_p) w_{\hat{s} \hat{t}}(e_{\hat{p}})c(u)c(e_p)c(e_s)c(e_t)c(w)c(e_{\hat{p}})c(e_{\hat{s}})c(e_{\hat{t}})\nonumber\\
		&-\frac{1}{16}\sum_{p,s,t,\hat{p},\hat{s},\hat{t}=1}^{2m} w_{st}(e_p) w_{\hat{s} \hat{t}}(e_{\hat{p}})c(u)c(e_p)c(e_s)c(e_t)c(w)c(e_{\hat{p}})\hat{c}(e_{\hat{s}})\hat{c}(e_{\hat{t}})\nonumber\\
		&-\frac{1}{16}\sum_{p,s,t,\hat{p},\hat{s},\hat{t}=1}^{2m} w_{st}(e_p) w_{\hat{s} \hat{t}}(e_{\hat{p}})c(u)c(e_{\hat{p}})\hat{c}(e_{\hat{s}})\hat{c}(e_{\hat{t}})c(w)c(e_p)c(e_s)c(e_t)\nonumber\\
		&+\frac{1}{16}\sum_{p,s,t,\hat{p},\hat{s},\hat{t}=1}^{2m} w_{st}(e_p) w_{\hat{s} \hat{t}}(e_{\hat{p}})c(u)c(e_p)\hat{c}(e_s)\hat{c}(e_t)c(w)c(e_{\hat{p}})\hat{c}(e_{\hat{s}})\hat{c}(e_{\hat{t}})\nonumber\\
&-\frac{1}{4}\sum_{p,s,t=1}^{2m} w_{st}(e_p) c(u)c(e_p)c(e_s)c(e_t)c(w)\hat{c}(V)+\frac{1}{4}\sum_{p,s,t=1}^{2m} w_{st}(e_p) c(u)c(e_p)\hat{c}(e_s)\hat{c}(e_t)c(w)\hat{c}(V)\nonumber\\
&-\frac{1}{4}\sum_{\hat{p},\hat{s},\hat{t}=1}^{2m} w_{\hat{s}\hat{t}}(e_{\hat{p}}) c(u)\hat{c}(V)c(w)c(e_{\hat{p}})c(e_{\hat{s}})c(e_{\hat{t}})+\frac{1}{4}\sum_{\hat{p},\hat{s},\hat{t}=1}^{2m} w_{\hat{s}\hat{t}}(e_{\hat{p}}) c(u)\hat{c}(V)c(w)c(e_{\hat{p}})\hat{c}(e_{\hat{s}})\hat{c}(e_{\hat{t}})\nonumber\\
&+\frac{1}{8}\sum_{j,p,s,t=1}^{2m} {\operatorname{R}}_{jpts}c(u)c(dx_j)c(w)c(e_p)c(e_s)c(e_t)-\frac{1}{8}\sum_{j,p,s,t=1}^{2m} {\operatorname{R}}_{jpts}c(u)c(dx_j)c(w)c(e_p)\hat{c}(e_s)\hat{c}(e_t)\nonumber\\
		&-\frac{1}{4}\sum_{p,s,t,j,\gamma=1}^{2m}w_{st}(e_p) \partial x_j(w_\gamma)c(u)c(dx_j)c(e_\gamma)c(e_p)c(e_s)c(e_t) \nonumber\\
&+\frac{1}{4}\sum_{p,s,t,j,\gamma=1}^{2m}w_{st}(e_p)\partial x_j(w_\gamma)c(u)c(dx_j)c(e_\gamma)c(e_p)\hat{c}(e_s)\hat{c}(e_t) \nonumber\\
		&+\sum_{p,s,t,j,\gamma=1}^{2m}\partial x_j(w_\gamma)c(u)c(dx_j)c(w)\hat{c}(V)+\sum_{p,s,t,j,\beta=1}^{2m}\partial x_j(V_\beta)c(u)c(dx_j)c(e_\gamma)\hat{c}(e_\beta)\nonumber\\
&+c(u)\hat{c}(V)c(w)\hat{c}(V);\nonumber\\
		\sigma_{1}(\mathcal{AB})=&\sigma_{1}(\mathcal{A}) \sigma_{0}(\mathcal{B})+\sigma_{0}(\mathcal{A}) \sigma_{1}(\mathcal{B})+(-i) \partial_{\xi_{j}}\left[\sigma_{1}(\mathcal{A})\right] \partial_{x_{j}}\left[\sigma_{1}(\mathcal{B})\right] \nonumber\\
		=&-\frac{i}{4}\sum_{p,s,t=1}^{2m} w_{st}(e_p) c(u)c(\xi)c(w)c(e_p)c(e_s)c(e_t)+\frac{i}{4}\sum_{p,s,t=1}^{2m} w_{st}(e_p) c(u)c(\xi)c(w)c(e_p)\hat{c}(e_s)\hat{c}(e_t)\nonumber\\
		&-\frac{i}{4}\sum_{p,s,t=1}^{2m} w_{st}(e_p) c(u)c(e_p)c(e_s)c(e_t)c(w)c(\xi)\nonumber+\frac{i}{4}\sum_{p,s,t=1}^{2m} w_{st}(e_p) c(u)c(e_p)\hat{c}(e_s)\hat{c}(e_t)c(w)c(\xi)\nonumber\\
		&+ic(u)c(\xi)c(w)\hat{c}(V)+ic(u)\hat{c}(V)c(w)c(\xi)+i\sum_{j,\gamma=1}^{2m}\partial x_j(w_\gamma)c(v)c(dx_j)c(e_\gamma)c(\xi);\nonumber\\
		\sigma_{2}(\mathcal{AB})=&\sigma_{1}(\mathcal{A})\sigma_{1}(\mathcal{B})=-c(u)c(\xi)c(w)c(\xi).\nonumber
	\end{align}
\end{lem}

\noindent {\bf (II-1)} For $\sigma_{0}(\mathcal{AB}) \sigma_{-2 m}\left(\widetilde{D}_v^{-2 m}\right)(x_{0})$:
\begin{align}\label{0-2m}
&\sigma_{0}(\mathcal{AB}) \sigma_{-2 m}\left(\widetilde{D}_v^{-2 m}\right)\left(x_{0}\right)\\
&=\frac{\|\xi\|^{-2 m}}{8}\sum_{j,p,s,t=1}^{2m}  {\operatorname{R}}_{jpts}c(u)c(dx_j)c(w)c(e_p)c(e_s)c(e_t)\nonumber\\
&-\frac{\|\xi\|^{-2 m}}{8}\sum_{j,p,s,t=1}^{2m}  {\operatorname{R}}_{jpts}c(u)c(dx_j)c(w)c(e_p)\hat{c}(e_s)\hat{c}(e_t)\nonumber\\
&+\|\xi\|^{-2 m}\sum_{j,\gamma=1}^{2m} \partial x_j(w_\gamma)c(u)c(dx_j)c(e_\gamma)\hat{c}(V)\nonumber\\
&+\|\xi\|^{-2 m}\sum_{j,\gamma=1}^{2m} \partial x_j(V_\beta)c(u)c(dx_j)c(w)\hat{c}(e_\beta)\nonumber\\
&+\|\xi\|^{-2 m}c(u)\hat{c}(V)c(w)\hat{c}(V).\nonumber
\end{align}

 \noindent{\bf (II-1-$\mathbb{A}$)}

\begin{align}\label{t2}
	&{\rm tr}\bigg(\sum_{j,p,t,s=1}^{2m}c(u)c(e_{j})c(w)c(e_p)c(e_s)c(e_t)\bigg)\\
	&=\sum_{j,p,t,s,r,f=1}^{2m}u_rw_f\bigg[-\delta_r^j\delta_f^p\delta_s^t+\delta_r^j\delta_f^s\delta_p^t-\delta_r^j\delta_f^t\delta_p^s+\delta_r^f\delta_p^s\delta_j^t-\delta_r^f\delta_p^t\delta_j^s+\delta_r^f\delta_p^j\delta_s^t-\delta_r^p\delta_j^f\delta_s^t\nonumber\\
&+\delta_r^p\delta_j^s\delta_f^t-\delta_r^p\delta_f^s\delta_j^t+\delta_r^s\delta_j^f\delta_p^t-\delta_r^s\delta_j^p\delta_f^t+\delta_r^s\delta_j^t\delta_f^p-\delta_r^t\delta_j^p\delta_f^s-\delta_r^t\delta_j^s\delta_f^p\bigg]{\rm tr}[id],\nonumber
\end{align}
then
\begin{align}
	&\int_{\|\xi\|=1}\operatorname{tr}\biggl\{ \frac{\|\xi\|^2}{8}\sum_{j,p,s,t=1}^{2m}  {\operatorname{R}}_{jpts}c(v)c(dx_j)c(w)c(e_p)c(e_s)c(e_t) \biggr\}(x_0)\sigma(\xi)\\
&=\bigg(\frac{1}{4}g(u,w)\sum_{ijp}R(e_j,e_p,e_j,e_p)-\frac{1}{2}\sum_pR(u,e_p,w,e_p)\bigg){\rm tr}[id]{\rm Vol}(S^{n-1})\nonumber\\
	&=\bigg(\frac{1}{4} s g(u,w)-\frac{1}{2}{\rm Ric}(u,w)\bigg){\rm tr}[id]{\rm Vol}(S^{n-1}).\nonumber
\end{align}

\noindent{\bf (II-1-$\mathbb{B}$)}

\begin{align}\label{qt3}
	&{\rm tr}\bigg(\sum_{j,p,t,s=1}^{2m}c(u)c(e_{j})c(w)c(e_p)\hat{c}(e_s)\hat{c}(e_t)\bigg)=\sum_{j,p,t,s,r,f=1}^{2m}u_rw_f\bigg[\delta_r^j\delta_f^p\delta_s^t-\delta_r^f\delta_j^p\delta_s^t+\delta_r^p\delta_j^f\delta_s^t\bigg]{\rm tr}[id],\nonumber
\end{align}
then
\begin{align}
	&\int_{\|\xi\|=1}\operatorname{tr}\biggl\{ -\frac{\|\xi\|^2}{8}\sum_{j,p,s,t=1}^{2m}  {\operatorname{R}}_{jpts}c(v)c(dx_j)c(w)c(e_p)\hat{c}(e_s)\hat{c}(e_t) \biggr\}(x_0)\sigma(\xi)\\
	&=0.\nonumber
\end{align}

\noindent{\bf (II-1-$\mathbb{C}$)}
\begin{align}
	&{\rm tr}\bigg(\sum_{j,\gamma=1}^{2m} \partial x_j(w_\gamma)c(u)c(dx_j)c(e_\gamma)\hat{c}(V)\bigg)=0,\nonumber
\end{align}
then
\begin{align}
&\int_{\|\xi\|=1}\|\xi\|^{-2 m}\sum_{j,\gamma=1}^{2m} \operatorname{tr}[\partial x_j(w_\gamma)c(u)c(dx_j)c(e_\gamma)\hat{c}(V)] \sigma(\xi)=0.\nonumber
\end{align}

\noindent{\bf (II-1-$\mathbb{D}$)}

Similar to {\bf (II-1-$\mathbb{C}$)}
\begin{align}
&\int_{\|\xi\|=1}\|\xi\|^{-2 m}\sum_{j,\gamma=1}^{2m} \operatorname{tr}[\partial x_j(V_\beta)c(u)c(dx_j)c(w)\hat{c}(e_\beta)] \sigma(\xi)=0.\nonumber
\end{align}

\noindent{\bf (II-1-$\mathbb{E}$)}

\begin{align}
	&{\rm tr}\bigg(c(u)\hat{c}(V)c(w)\hat{c}(V)\bigg)=|V|^2g(u,w){\rm tr}[id],\nonumber
\end{align}
then
\begin{align}
&\int_{\|\xi\|=1}\|\xi\|^{-2 m}c(u)\hat{c}(V)c(w)\hat{c}(V)=|V|^2g(u,w){\rm tr}[id]{\rm Vol}(S^{n-1}).\nonumber
\end{align}

Therefore, we get
\begin{align}
	&\int_{\|\xi\|=1} \operatorname{tr}\bigg[\sigma_{0}(\mathcal{AB}) \sigma_{-2 m}\left(\widetilde{D}_v^{-2 m}\right)\left(x_{0}\right)\bigg] \sigma(\xi)\\
	&=\bigg(\frac{1}{4} s g(u,w)-\frac{1}{2}{\rm Ric}(u,w)+|V|^2g(u,w)\bigg){\rm tr}[id]{\rm Vol}(S^{n-1}).\nonumber
\end{align}

\noindent{\bf (II-2)} For $\sigma_{1}(\mathcal{AB}) \sigma_{-2 m-1}\left(\widetilde{D}_v^{-2 m}\right)\left(x_{0}\right)$:
\begin{align}
	&\sigma_{1}(\mathcal{AB}) \sigma_{-2 m-1}\left(\widetilde{D}_v^{-2 m}\right)\left(x_{0}\right)=0.\nonumber
\end{align}

Therefore, we get
\begin{align}
	&\int_{\|\xi\|=1} \operatorname{tr}\bigg[\sigma_{1}(\mathcal{AB}) \sigma_{-2 m-1}\left(\widetilde{D}_v^{-2 m}\right)\left(x_{0}\right)\bigg] \sigma(\xi)=0.\nonumber
\end{align}

\noindent{\bf (II-3)} For $\sigma_{2}(\mathcal{AB}) \sigma_{-2 m-2}\left(\widetilde{D}_v^{-2 m}\right)$:
\begin{align}\label{2-2m-2}
	&\sigma_{2}(\mathcal{AB}) \sigma_{-2 m-2}\left(\widetilde{D}_v^{-2 m}\right)\left(x_{0}\right)\\
	&=-\frac{m(m+1)}{3}\|\xi\|^{-2 m-4}\sum_{a,b=1}^{2m} \operatorname{Ric}_{a b} \xi_{a} \xi_{b}c(u)c(\xi)c(w)c(\xi) \nonumber\\
	&\;\;\;\;+\frac{ m (m+1)}{4}\|\xi\|^{-2 m-4} \sum_{a,b,t,s=1}^{2m}\operatorname{R}_{b a t s} c(u)c(\xi)c(w)c(\xi)c(e_s) c(e_t) \xi_{a} \xi_{b}\nonumber\\
&\;\;\;\;-\frac{ m (m+1)}{4}\|\xi\|^{-2 m-4} \sum_{a,b,t,s=1}^{2m}\operatorname{R}_{b a t s} c(u)c(\xi)c(w)c(\xi)\hat{c}(e_s) \hat{c}(e_t) \xi_{a} \xi_{b}\nonumber\\
&\;\;\;\;+\frac{ m }{8}\|\xi\|^{-2 m-2}\sum_{i,j,k,l=1}^{2m}c(u)c(\xi)c(w)c(\xi)\hat{c}(e_i) \hat{c}(e_j)c(e_k)c(e_l)\nonumber\\
	&\;\;\;\;+\frac{m}{4}\|\xi\|^{-2 m-2}c(u)c(\xi)c(w)c(\xi)s\nonumber\\
	&\;\;\;\;+m\|\xi\|^{-2 m-2}\sum_{i=1}^{2m}c(u)c(\xi)c(w)c(\xi)c(e_i)\hat{c}(\nabla^{TM}_{e_i}V)\nonumber\\
&\;\;\;\;+m\|\xi\|^{-2 m-2}|V|^2c(u)c(\xi)c(w)c(\xi).\nonumber
\end{align}

\noindent{\bf (II-3-$\mathbb{A}$)}
\begin{align}\label{t7}
	{\rm tr}[c(u)c(\xi)c(w)c(\xi)]=\sum_{f,g=1}^{2m}\xi_f\xi_g{\rm tr}[c(u)c(e_f)c(w)c(e_g)],
\end{align}
then
\begin{align}
	&\int_{\|\xi\|=1} \operatorname{tr}\biggl\{-\frac{m(m+1)}{3}\|\xi\|^{-2 m-4}\sum_{a,b=1}^{2m} \operatorname{Ric}_{a b} \xi_{a} \xi_{b} c(u)c(\xi)c(w)c(\xi)\biggr\}\sigma(\xi)\nonumber\\
	&=-\frac{m(m+1)}{3}\times\frac{1}{2m(2m+2)}\operatorname{Ric}_{a b}(\delta_b^a\delta_f^g+\delta_a^f\delta_b^g+\delta_a^g\delta_b^f){\rm tr}[c(u)c(e_f)c(w)c(e_g)]{\rm Vol}(S^{n-1})\nonumber\\
&=-\frac{1}{12}\sum_{laf}R(e_l,e_a,e_l,e_a){\rm tr}[c(u)c(e_f)c(w)c(e_f)]{\rm Vol}(S^{n-1})\nonumber\\
&-\frac{1}{12}\sum_{lab}R(e_l,e_a,e_l,e_b){\rm tr}[c(u)c(e_a)c(w)c(e_b)]{\rm Vol}(S^{n-1})\nonumber\\
&-\frac{1}{12}\sum_{laf}R(e_l,e_a,e_l,e_b){\rm tr}[c(u)c(e_b)c(w)c(e_a)]{\rm Vol}(S^{n-1})\nonumber\\
	&=\;\;\bigg(\frac{m}{6} s g(u,w)-\frac{1}{3}{\rm Ric}(u,w)\bigg){\rm tr}[id]{\rm Vol}(S^{n-1}).
\end{align}

\noindent{\bf (II-3-$\mathbb{B}$)}

\begin{align}
	&\int_{\|\xi\|=1}\operatorname{tr} \biggl\{\frac{ m (m+1)}{4}\|\xi\|^{-2 m-4} \sum_{a,b,t,s=1}^{2m}\operatorname{R}_{b a t s} c(u)c(\xi)c(w)c(\xi)c(e_s) c(e_t) \xi_{a} \xi_{b}\biggr\}\sigma(\xi)\nonumber\\
&=\frac{ m (m+1)}{4}\times\frac{1}{2m(2m+2)}\sum_{a,b,f,g,t,s=1}^{2m}\operatorname{R}_{b a t s} (\delta_b^a\delta_f^g+\delta_a^f\delta_b^g+\delta_a^g\delta_b^f) {\rm tr}[c(u)c(e_a)c(w)c(e_b)c(e_s) c(e_t)]{\rm Vol}(S^{n-1})\nonumber\\
&=\frac{1}{16}\sum_{a,b,t,s=1}^{2m}\operatorname{R}_{b a t s} {\rm tr}[c(u)c(e_a)c(w)c(e_b)c(e_s) c(e_t)+c(u)c(e_b)c(w)c(e_a)c(e_s) c(e_t)]{\rm Vol}(S^{n-1}) \nonumber\\
	&=0.\nonumber
\end{align}

\noindent{\bf (II-3-$\mathbb{C}$)}

\begin{align}
	&\int_{\|\xi\|=1}\operatorname{tr} \biggl\{\frac{ m (m+1)}{4}\|\xi\|^{-2 m-4} \sum_{a,b,t,s=1}^{2m}\operatorname{R}_{b a t s} c(u)c(\xi)c(w)c(\xi)\hat{c}(e_s) \hat{c}(e_t) \xi_{a} \xi_{b}\biggr\}\sigma(\xi)\nonumber\\
&=\frac{ m (m+1)}{4}\times\frac{1}{2m(2m+2)}\sum_{a,b,f,g,t,s=1}^{2m}\operatorname{R}_{b a t s} (\delta_b^a\delta_f^g+\delta_a^f\delta_b^g+\delta_a^g\delta_b^f) {\rm tr}[c(u)c(e_a)c(w)c(e_b)\hat{c}(e_s) \hat{c}(e_t)] {\rm Vol}(S^{n-1})\nonumber\\
&=\frac{1}{16}\sum_{a,b,t,s=1}^{2m}\operatorname{R}_{b a t s} {\rm tr}[c(u)c(e_a)c(w)c(e_b)\hat{c}(e_s) \hat{c}(e_t)+c(u)c(e_b)c(w)c(e_a)\hat{c}(e_s) \hat{c}(e_t)] {\rm Vol}(S^{n-1}) \nonumber\\
	&=0.\nonumber
\end{align}

\noindent{\bf (II-3-$\mathbb{D}$)}

\begin{align}
	&\int_{\|\xi\|=1}\operatorname{tr} \biggl\{\frac{ m }{8}\|\xi\|^{-2 m-2}\sum_{i,j,k,l=1}^{2m}c(u)c(\xi)c(w)c(\xi)\hat{c}(e_i) \hat{c}(e_j)c(e_k)c(e_l)\biggr\}\sigma(\xi)\nonumber\\
	&=\int_{\|\xi\|=1}\operatorname{tr} \biggl\{\frac{ m }{8}\|\xi\|^{-2 m-2}\sum_{i,j,k,l=1}^{2m}c(u)c(e_f)c(w)c(e_g)\hat{c}(e_i) \hat{c}(e_j)c(e_k)c(e_l)\xi_{f} \xi_{g}\biggr\}\sigma(\xi)\nonumber\\
	&=\;\;\frac{ 1}{16}\sum_{i,j,k,l=1}^{2m}\operatorname{R}_{ijkl} {\rm tr}[c(u)c(e_f)c(w)c(e_g)\hat{c}(e_i) \hat{c}(e_j)c(e_k)c(e_l)]{\rm Vol}(S^{n-1})\nonumber\\
	&=\;\;0.\nonumber
\end{align}

\noindent{\bf (II-3-$\mathbb{E}$)}

\begin{align}
	&\int_{\|\xi\|=1}\operatorname{tr} \biggl\{\frac{m}{4}\|\xi\|^{-2 m-2}c(u)c(\xi)c(w)c(\xi)s\biggr\}\sigma(\xi)\nonumber\\
	&=\int_{\|\xi\|=1}\operatorname{tr} \biggl\{\frac{m}{4}\|\xi\|^{-2 m-2}c(u)c(e_f)c(w)c(e_g)s\xi_{f} \xi_{g}\biggr\}\sigma(\xi)\nonumber\\
	&=\;\;\frac{ 1}{8}s{\rm tr}[c(u)c(e_f)c(w)c(e_f)]{\rm Vol}(S^{n-1}) \nonumber\\
	&=\;\;-\frac{1}{4}(m-1)g(u,w){\rm tr}[id]{\rm Vol}(S^{n-1}).
\end{align}

\noindent{\bf (II-3-$\mathbb{F}$)}

\begin{align}
&\int_{\|\xi\|=1}\operatorname{tr} \biggl\{m\|\xi\|^{-2 m-2}\sum_{i=1}^{2m}c(u)c(\xi)c(w)c(\xi)c(e_i)\hat{c}(\nabla^{TM}_{e_i}V)\biggr\}\sigma(\xi)\nonumber\\
&=\int_{\|\xi\|=1}\operatorname{tr} \biggl\{m\|\xi\|^{-2 m-2}\sum_{i,f,g=1}^{2m}c(u)c(e_f)c(w)c(e_g)c(e_i)\hat{c}(\nabla^{TM}_{e_i}V)\xi_{f}\xi_g\biggr\}\sigma(\xi)\nonumber\\
&=\;\;\frac{1}{2}\sum_{i,f=1}^{2m}{\rm tr}\bigg(c(u)c(e_f)c(w)c(e_f)c(e_i)\hat{c}(\nabla^{TM}_{e_i}V)\bigg){\rm Vol}(S^{n-1})\nonumber\\
&=0.\nonumber
\end{align}

\noindent{\bf (II-3-$\mathbb{G}$)}

\begin{align}
	&\int_{\|\xi\|=1}\operatorname{tr} \biggl\{m\|\xi\|^{-2 m-2}|V|^2c(u)c(\xi)c(w)c(\xi)\biggr\}\sigma(\xi)\nonumber\\
	&=\int_{\|\xi\|=1} \operatorname{tr}\biggl\{m\|\xi\|^{-2 m-2}|V|^2c(u)c(e_f)c(w)c(e_g)\xi_{f}\xi_{g}\biggr\}\sigma(\xi)\nonumber\\
	&=\;\;\frac{1}{2}|V|^2{\rm tr}[c(u)c(e_f)c(w)c(e_f)]{\rm Vol}(S^{n-1})\nonumber\\
	&=\;\;-(m-1)||V|g(u,w){\rm tr}[id]{\rm Vol}(S^{n-1}).\nonumber
\end{align}

Therefore, we get
\begin{align}
	&\int_{\|\xi\|=1} \operatorname{tr}\bigg[\sigma_{2}(\mathcal{AB}) \sigma_{-2 m-2}\left(\widetilde{D}_v^{-2 m}\right)\bigg] \sigma(\xi)\\
	&=\;\;\bigg(\frac{3-m}{12} s g(u,w)-\frac{1}{3}{\rm Ric}(u,w)-(m-1)|V|^2g(u,w)\bigg){\rm tr}[id]{\rm Vol}(S^{n-1}).\nonumber
\end{align}

\noindent{\bf (II-4)} For $-i\sum_{j=1}^{2m} \partial_{\xi_{j}}\left[\sigma_{2}(\mathcal{A} \mathcal{B})\right] \partial_{x_{j}}\left[\sigma_{-2 m-1}\left(\widetilde{D}_v^{-2 m}\right)\right]$:

\begin{align}\label{2-2m-1}
	&-i\sum_{j=1}^{2m} \partial_{\xi_{j}}\left[\sigma_{2}(\mathcal{A} \mathcal{B})\right] \partial_{x_{j}}\left[\sigma_{-2 m-1}\left(\widetilde{D}_v^{-2 m}\right)\right]\\
	&=\;\;\frac{2 m }{3}\|\xi\|^{-2 m-2}\sum_{a,b,j=1}^{2m}  \operatorname{Ric}_{a b}  \xi_{a}\delta^{b}_{j} \bigg(c(u)c(dx_j)c(w)c(\xi)+c(u)c(\xi)c(w)c(dx_j)\bigg)\nonumber\\
	&\;\;\;\;-\frac{ m }{4}\|\xi\|^{-2 m-2} \sum_{a,b,j,t,s=1}^{2m} \operatorname{R}_{b a t s}\xi_{a}\delta^{b}_{j}\bigg(c(u)c(dx_j)c(w)c(\xi)+c(u)c(\xi)c(w)c(dx_j) \bigg) c(e_s) c(e_t) \nonumber\\
	&\;\;\;\;+\frac{ m }{4}\|\xi\|^{-2 m-2} \sum_{a,b,j,t,s=1}^{2m} \operatorname{R}_{b a t s}\xi_{a}\delta^{b}_{j}\bigg(c(u)c(dx_j)c(w)c(\xi) +c(u)c(\xi)c(w)c(dx_j) \bigg)\hat{c}(e_s) \hat{c}(e_t).  \nonumber\\
\end{align}

\noindent{\bf (II-4-$\mathbb{A}$)}

\begin{align}
	&\int_{\|\xi\|=1} {\rm tr}\biggl\{\frac{2 m }{3}\|\xi\|^{-2 m-2}\sum_{a,b,j=1}^{2m}  \operatorname{Ric}_{a b}  \xi_{a}\delta^{b}_{j}\bigg(c(u)c(dx_j)c(w)c(\xi)+c(u)c(\xi)c(w)c(dx_j)\bigg)\biggr\}(x_0)\sigma(\xi)\nonumber\\
	&=\int_{\|\xi\|=1}{\rm tr} \biggl\{\frac{2 m }{3}\|\xi\|^{-2 m-2}\sum_{a,b,f=1}^{2m}  \operatorname{Ric}_{a b}  \xi_{a}\xi_{f}\bigg(c(u)c(e_b)c(w)c(e_f)+c(u)c(e_f)c(w)c(e_b)\bigg)\biggr\}(x_0)\sigma(\xi)\nonumber\\
	&=\;\;\frac{1}{3}\sum_{a,b=1}^{2m}  \operatorname{Ric}_{a b}  {\rm tr}\bigg(c(u)c(e_b)c(w)c(e_a)+c(u)c(e_a)c(w)c(e_b)\bigg)(x_0){\rm Vol}(S^{n-1})\nonumber\\
	&=\;\;\frac{2}{3}\sum_{a,b=1}^{2m}  \operatorname{Ric}_{a b}  {\rm tr}\bigg(c(u)c(e_b)c(w)c(e_a)\bigg)(x_0){\rm Vol}(S^{n-1})\nonumber\\
	&=\;\;\bigg(\frac{4}{3}\operatorname{Ric}(u,w)-\frac{2}{3}s g(u,w)\bigg){\rm tr}[id] {\rm Vol}(S^{n-1}).\nonumber
\end{align}

\noindent{\bf (II-4-$\mathbb{B}$)}

\begin{align}
	&\int_{\|\xi\|=1}\operatorname{tr} \biggl\{-\frac{ m }{4}\|\xi\|^{-2 m-2} \sum_{a,b,j,t,s=1}^{2m} \operatorname{R}_{b a t s}\xi_{a}\delta^{b}_{j}\bigg(c(u)c(dx_j)c(w)c(\xi) +c(u)c(\xi)c(w)c(dx_j)\bigg) c(e_s) c(e_t)\biggr\}(x_0)\sigma(\xi)\nonumber\\
	&=\int_{\|\xi\|=1}\operatorname{tr} \biggl\{-\frac{ m }{4}\|\xi\|^{-2 m-2} \sum_{a,b,t,s,f=1}^{2m} \operatorname{R}_{b a t s}\xi_{a}\xi_{f}\bigg(c(u)c(e_b)c(w)c(e_f) +c(u)c(e_f)c(w)c(e_b) \bigg)c(e_s) c(e_t)\biggr\}(x_0)\sigma(\xi)\nonumber\\
&=-\frac{1 }{8}\sum_{a,b,t,s=1}^{2m} \operatorname{R}_{b a t s}\operatorname{tr}\bigg(c(u)c(e_b)c(w)c(e_a) c(e_s) c(e_t)\bigg){\rm Vol}(S^{n-1})\nonumber\\
&-\frac{1 }{8}\sum_{a,b,t,s=1}^{2m} \operatorname{R}_{ab t s}\operatorname{tr}\bigg(c(u)c(e_b)c(w)c(e_a)c(e_s) c(e_t)\bigg){\rm Vol}(S^{n-1})\nonumber\\
	&=\;\;0.\nonumber
\end{align}

\noindent{\bf (II-4-$\mathbb{C}$)}

\begin{align}
	&\int_{\|\xi\|=1}\operatorname{tr} \biggl\{-\frac{ m }{4}\|\xi\|^{-2 m-2} \sum_{a,b,j,t,s=1}^{2m} \operatorname{R}_{b a t s}\xi_{a}\delta^{b}_{j}\bigg(c(u)c(dx_j)c(w)c(\xi) +c(u)c(\xi)c(w)c(dx_j)\bigg) \hat{c}(e_s)\hat{ c}(e_t)\biggr\}(x_0)\sigma(\xi)\nonumber\\
	&=\int_{\|\xi\|=1}\operatorname{tr} \biggl\{-\frac{ m }{4}\|\xi\|^{-2 m-2} \sum_{a,b,t,s,f=1}^{2m} \operatorname{R}_{b a t s}\xi_{a}\xi_{f}\bigg(c(u)c(e_b)c(w)c(e_f) +c(u)c(e_f)c(w)c(e_b) \bigg)\hat{c}(e_s)\hat{ c}(e_t)\biggr\}(x_0)\sigma(\xi)\nonumber\\
&=-\frac{1 }{8}\sum_{a,b,t,s=1}^{2m} \operatorname{R}_{b a t s}\operatorname{tr}\bigg(c(u)c(e_b)c(w)c(e_a)\hat{c}(e_s)\hat{ c}(e_t)\bigg) {\rm Vol}(S^{n-1})\nonumber\\
&-\frac{1 }{8}\sum_{a,b,t,s=1}^{2m} \operatorname{R}_{ab t s}\operatorname{tr}\bigg(c(u)c(e_b)c(w)c(e_a)\hat{c}(e_s)\hat{ c}(e_t)\bigg) {\rm Vol}(S^{n-1})\nonumber\\
	&=\;\;0.\nonumber
\end{align}

Therefore, we get
\begin{align}
	&\int_{\|\xi\|=1} \operatorname{tr}\bigg[-i \sum_{j=1}^{2m} \partial_{\xi_{j}}\left[\sigma_{2}(\mathcal{A} \mathcal{B})\right] \partial_{x_{j}}\left[\sigma_{-2 m-1}\left(\widetilde{D}_v^{-2 m}\right)\right]\bigg] (x_0)\sigma(\xi)\\
	&=\;\;\bigg(\frac{4}{3}\operatorname{Ric}(u,w)-\frac{2}{3}s g(u,w)\bigg){\rm tr}[id] {\rm Vol}(S^{n-1}).\nonumber
\end{align}

\noindent{\bf (II-5)} For $-i \sum_{j=1}^{2m} \partial_{\xi_{j}}\left[\sigma_{1}(\mathcal{A} \mathcal{B})\right] \partial_{x_{j}}\left[\sigma_{-2 m}\left(\widetilde{D}_v^{-2 m}\right)\right]$:
\begin{align}\label{1-2m}
	&\partial_{x_{j}}\left[\sigma_{-2 m}\left(\widetilde{D}_v^{-2 m}\right)\right](x_0)\\
	&=\partial_{x_{j}}\bigg[\|\xi\|^{-2 m-2}\sum_{a,b,l,k=1}^{2m} \left(\delta_{a b}-\frac{m}{3} R_{a l b k} x^{l} x^{k}\right) \xi_{a} \xi_{b}\bigg](x_0)\nonumber\\
	&=\;\;0,\nonumber
\end{align}
then
\begin{align}
	&\int_{\|\xi\|=1} \operatorname{tr}\bigg[-i\sum_{j=1}^{2m} \partial_{\xi_{j}}\left[\sigma_{1}(\mathcal{A} \mathcal{B})\right] \partial_{x_{j}}\left[\sigma_{-2 m}\left(\widetilde{D}_v^{-2 m}\right)\right]\bigg] (x_0)\sigma(\xi)=0.\nonumber
\end{align}

\noindent{\bf (II-6)} For $-\frac{1}{2} \sum_{j ,l=1}^{2m} \partial_{\xi_{j}} \partial_{\xi_{l}}\left[\sigma_{2}(\mathcal{A} \mathcal{B})\right] \partial_{x_{j}} \partial_{x_{l}}\left[\sigma_{-2 m}\left(\widetilde{D}_v^{-2 m}\right)\right]$:

\begin{align}
&\sum_{j, l=1} ^{2m}\partial_{\xi_{j}} \partial_{\xi_{l}}\left[\sigma_{2}(\mathcal{A} \mathcal{B})\right]\nonumber\\
&=\sum_{j, l=1} ^{2m}\partial_{\xi_{j}} \partial_{\xi_{l}}\left[c(u)c(\xi)c(w)c(\xi)\right]\nonumber\\
&=\sum_{j, l=1} ^{2m}[-c(u)c(dx_l)c(w)c(dx_j)-c(u)c(dx_j)c(w)c(dx_l)]
\end{align}

\begin{align}
&\sum_{j, l=1} ^{2m}\partial_{x_{j}} \partial_{x_{l}}\left[\sigma_{-2 m}\left(\widetilde{D}_v^{-2 m}\right)\right]\nonumber\\
&=\sum_{j, l=1} ^{2m}\partial_{x_{j}} \partial_{x_{l}}\left[\|\xi\|^{-2 m-2}\sum_{a,b,\hat{j},k=1}^{2m} \left(\delta_{a b}-\frac{m}{3} R_{a \hat{j} b k} x^{\hat{j}} x^{k}\right) \xi_{a} \xi_{b}\right]\nonumber\\
&=-\frac{m}{3}\sum_{j, l,a,b,\hat{j},k=1} ^{2m}\|\xi\|^{-2 m-2}\bigg(R_{a\hat{j}bk}\delta_l^{\hat{j}}\delta_j^k-R_{a\hat{j}bk}\delta_l^k\delta_j^{\hat{j}}\bigg)\xi_a\xi_b.
\end{align}

\begin{align}\label{1-2m}
	&-\frac{1}{2} \sum_{j, l=1} ^{2m}\partial_{\xi_{j}} \partial_{\xi_{l}}\left[\sigma_{2}(\mathcal{A} \mathcal{B})\right] \partial_{x_{j}} \partial_{x_{l}}\left[\sigma_{-2 m}\left(\widetilde{D}_v^{-2 m}\right)\right]\\
	&=-\frac{m}{6}\|\xi\|^{-2 m-2}\sum_{j, l,a,b,\hat{j},k=1} ^{2m}\|\xi\|^{-2 m-2}\bigg(R_{a\hat{j}bk}\delta_l^{\hat{j}}\delta_j^k-R_{a\hat{j}bk}\delta_l^k\delta_j^{\hat{j}}\bigg)c(u)c(dx_l)c(w)c(dx_j)\xi_a\xi_b\nonumber\\
&-\frac{m}{6}\|\xi\|^{-2 m-2}\sum_{j, l,a,b,\hat{j},k=1} ^{2m}\|\xi\|^{-2 m-2}\bigg(R_{a\hat{j}bk}\delta_l^{\hat{j}}\delta_j^k-R_{a\hat{j}bk}\delta_l^k\delta_j^{\hat{j}}\bigg)c(u)c(dx_j)c(w)c(dx_l)\xi_a\xi_b.\nonumber
\end{align}

then
\begin{align}
	&\int_{\|\xi\|=1} {\rm tr}\biggl\{-\frac{1}{2} \sum_{j, l=1}^{2m} \partial_{\xi_{j}} \partial_{\xi_{l}}\left[\sigma_{2}(\mathcal{A} \mathcal{B})\right] \partial_{x_{j}} \partial_{x_{l}}\left[\sigma_{-2 m}\left(\widetilde{D}_v^{-2 m}\right)\right]\biggr\}(x_0)\sigma(\xi)\nonumber\\
	&=\int_{\|\xi\|=1} \biggl\{-\frac{m}{6}\|\xi\|^{-2 m-2}\sum_{a,b,j,l=1}^{2m}\bigg({\rm R}_{albj}+{\rm R}_{ajbl} \bigg){\rm tr}[c(u)c(dx_l)c(w)c(dx_j)]\xi_{a}\xi_{b}\biggr\}(x_0)\sigma(\xi)\nonumber\\
&+\int_{\|\xi\|=1} \biggl\{-\frac{m}{6}\|\xi\|^{-2 m-2}\sum_{a,b,j,l=1}^{2m}\bigg({\rm R}_{albj}+{\rm R}_{ajbl} \bigg){\rm tr}[c(u)c(dx_j)c(w)c(dx_l)]\xi_{a}\xi_{b}\biggr\}(x_0)\sigma(\xi)\nonumber\\
	&=-\frac{1}{6}\sum_{a,j,l=1}^{2m}{\rm R}_{alaj}{\rm tr}[c(u)c(dx_l)c(w)c(dx_j)+c(u)c(dx_j)c(w)c(dx_l)](x_0){\rm Vol}(S^{n-1})\nonumber\\
	&=\bigg(\frac{1}{3}s g(u,w)-\frac{2}{3}{\rm Ric}(u,w)\bigg){\rm tr}[id] {\rm Vol}(S^{n-1}).\nonumber
\end{align}

Thus, we get

\begin{align}\label{zabdt}
	&\int_{\|\xi\|=1} {\rm tr}\biggl\{	\sigma_{-2 m}\left(\mathcal{A} \mathcal{B} \widetilde{D}_v\widetilde{D}_v^{-2 m}\right)\biggr\}(x_0)\sigma(\xi)\\
	&=\bigg(\frac{2-m}{12} s g(u,w)+(2-m)|V|^2g(u,w)-\frac{1}{6}{\rm Ric}(u,w)\bigg){\rm tr}[id] {\rm Vol}(S^{n-1}).\nonumber
\end{align}
Further, we obtain
\begin{align}\label{z2}
	\mathscr{S}_{2}=&\mathrm{Wres}\bigg(c(u)\widetilde{D}_vc(w) \widetilde{D}_v \widetilde{D}_v^{-2m}\bigg)\\
	=&\;2^{2m} \frac{2 \pi^{m}}{\Gamma\left(m\right)}\int_{M}\biggl\{\frac{2-m}{12} s g(u,w)+(2-m)|V|^2g(u,w)-\frac{1}{6}{\rm Ric}(u,w)\biggr\}d{\rm Vol}_M.\nonumber
\end{align}
Hence, by summing $\mathscr{S}_{1}$ and $\mathscr{S}_{2}$, Theorem \ref{thm} holds.

\end{proof}

\section*{ Declarations}
\textbf{Ethics approval and consent to participate:} Not applicable.

\textbf{Consent for publication:} Not applicable.

\textbf{Availability of data and materials:} The authors confrm that the data supporting the findings of this study are available within the article.

\textbf{Competing interests:} The authors declare no competing interests.

\textbf{Author Contributions:} All authors contributed to the study conception and design. Material preparation,
data collection and analysis were performed by TW and YW. The first draft of the manuscript was written
by TW and all authors commented on previous versions of the manuscript. All authors read and approved
the final manuscript.
\section*{Acknowledgements}
 The authors were supported by Science and Technology Development Plan Project of Jilin Province, China: No.20260102245JC, NSFC. No.12401059 and Liaoning Province Science and Technology Plan Joint Project 2023-BSBA-118. The authors thank the referee for his (or her) careful reading and helpful comments.

\end{document}